\documentclass{article}
\usepackage{graphicx,hyperref}
\usepackage{amsmath,amssymb,amsthm}
\usepackage[utf8]{inputenc}

\providecommand{\keywords}[1]
{
  \small	
  \textbf{\textit{Keywords---}} #1
}

\usepackage[top=1.2in,bottom=1.2in,left=1.2in,right=1.2in]{geometry}
\newtheorem{theorem}[]{Theorem}

\title{Approximating monomials using Chebyshev polynomials}
\author{Arvind K.\ Saibaba\footnote{Department of Mathematics, North Carolina State University, Raleigh, NC. Email: \url{asaibab@ncsu.edu}.}}
\date{\today}

\begin{document}

\maketitle
\begin{abstract}
    This paper considers the approximation of a monomial $x^n$ over the interval $[-1,1]$ by a lower-degree polynomial. This polynomial approximation can be easily computed analytically and is obtained by truncating the analytical Chebyshev series expansion of $x^n$. The error in the polynomial approximation in the supremum norm has an exact expression with an interesting probabilistic interpretation. We use this interpretation along with concentration inequalities to develop a useful upper bound for the error.  
\end{abstract}

\section{Motivation and Introduction}\label{sec:intro}
We are interested in approximating the monomial $x^n$ by a polynomial of degree $0 \leq k < n$ over the interval $[-1,1]$. The monomials $1,x,x^2,\dots$ form a basis for $C[-1,1]$, so it seems unlikely that we can represent a monomial in terms of lower degree polynomials. In Figure~\ref{fig:monomialbasis}, we plot a few functions from the monomial basis over $[0,1]$; the basis function look increasingly alike as we take higher and higher powers, i.e., they appear to ``lose independence.'' Numerical analysts often avoid the monomial basis in polynomial interpolation since they result in ill-conditioned Vandermonde matrices, leading to poor numerical performance in finite precision arithmetic. This loss of independence means that it is reasonable to approximate the monomial $x^n$ as a linear combination of lower order monomials, i.e., a lower order polynomial approximation. The natural question to ask, therefore, is: how small can $k$ be so that a well-chosen polynomial of degree $k$ can accurately approximate $x^n$? 

\begin{figure}[!ht]
    \centering
    \includegraphics[scale=0.4]{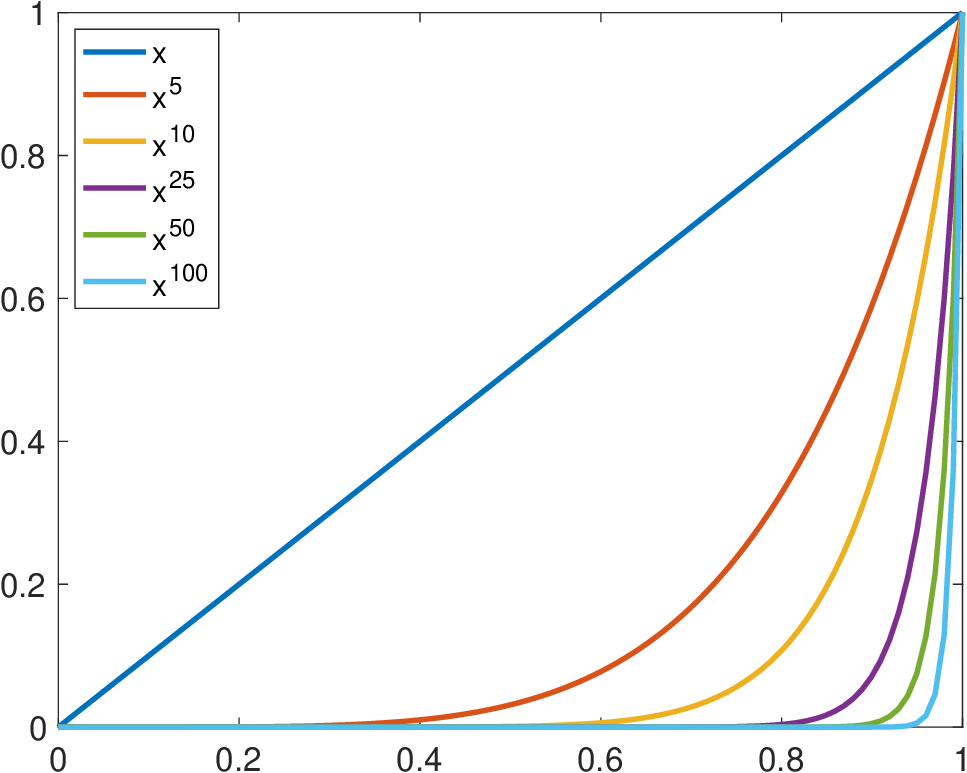}
    \caption{Visualization of a few monomials in the interval $[0,1]$.}
    \label{fig:monomialbasis}
\end{figure}

The surprising answer to this question is that we can approximate the monomial $x^n$ over $[-1,1]$ by a polynomial of small degree, which we will make precise. Let $\|f\|_\infty = \max_{x \in [-1,1]} | f(x)|$ denote the supremum norm on $C[-1,1]$  and let  $\pi_k^*(\cdot)$ be the best polynomial approximation to $x^n$ in this norm; that is
\[ E_{n,k} := \min_{\pi \in \mathcal{P}_k} \|x^n- \pi(x) \|_\infty= \|x^n - \pi_k^*(x)\|_\infty,   \]
where $\mathcal{P}_k$ is a vector space of polynomials with real coefficients of degree at most $k$.  The minimizer $\pi_k^*(\cdot)$ exists and is unique~\cite[Chapter 10]{trefethen2013approximation}, but does not have a closed form expression. Newman and Rivlin~\cite[Theorem 2]{newman1976approximation} showed that\footnote{We briefly mention that the notation our manuscript differs from~\cite{newman1976approximation} in that we reverse the roles of $n$ and $k$.}  
\begin{equation}\label{eqn:newmanrivlin} \frac{p_{n,k}}{4e} \leq \|x^n -  \pi_k^*(x)\|_\infty \leq p_{n,k},  \end{equation}
where the term $p_{n,k}$ is given by the formula  \[ p_{n,k}  = \frac{1}{2^{n-1}} \sum_{j = \lfloor (n+k)/2 \rfloor +1}^n \binom{n}{j}. \]
Since $p_{n,k}$ involves the sum of binomial coefficients, it has a probabilistic interpretation which we explore in Section~\ref{sec:prob}.

To see why a small $k$ is sufficient, consider the upper bound $p_{n,k}$. In Section~\ref{sec:prob} we use the probabilistic interpretation to obtain the following bound  $p_{n,k} \leq 2\exp\left(-{k^2}/{2n}\right)$. Suppose we are given a user-defined tolerance $\epsilon > 0$. To ensure   
\[ \|x^n -  \pi_k^*(x)\|_\infty \leq \epsilon,\]
we need to  choose $k \geq \sqrt{2n\log(2/\epsilon)}$. The accuracy of the polynomial approximation is visualized in Figure~\ref{fig:approx}, where in the left panel we plot the monomial $x^{n}$ for $n=75$ and the best polynomial approximation $\pi_{k}^*$ for $k=5,15,25$. The polynomial $\pi_k^*$ is computed using the Remez algorithm, implemented in chebfun~\cite{Driscoll2014}. We see that for $k=25$, the polynomial approximation looks very accurate. In the right panel, we display $p_{n,k}$, which is the upper bound of the best polynomial approximation, as well as the upper bound for $p_{n,k}$. We see that $p_{n,k}$ and its upper bound both have sharp decay with increasing $k$.  Numerical evidence in~\cite{nakatsukasa2018rational} further confirms this analysis; the authors show that the error $E_{n,k}$ behaves approximately like $\frac12 \text{erfc}(k/\sqrt{n})$, where erfc is the complementary error function. .

\begin{figure}[!ht]
    \centering
    \includegraphics[scale=0.5]{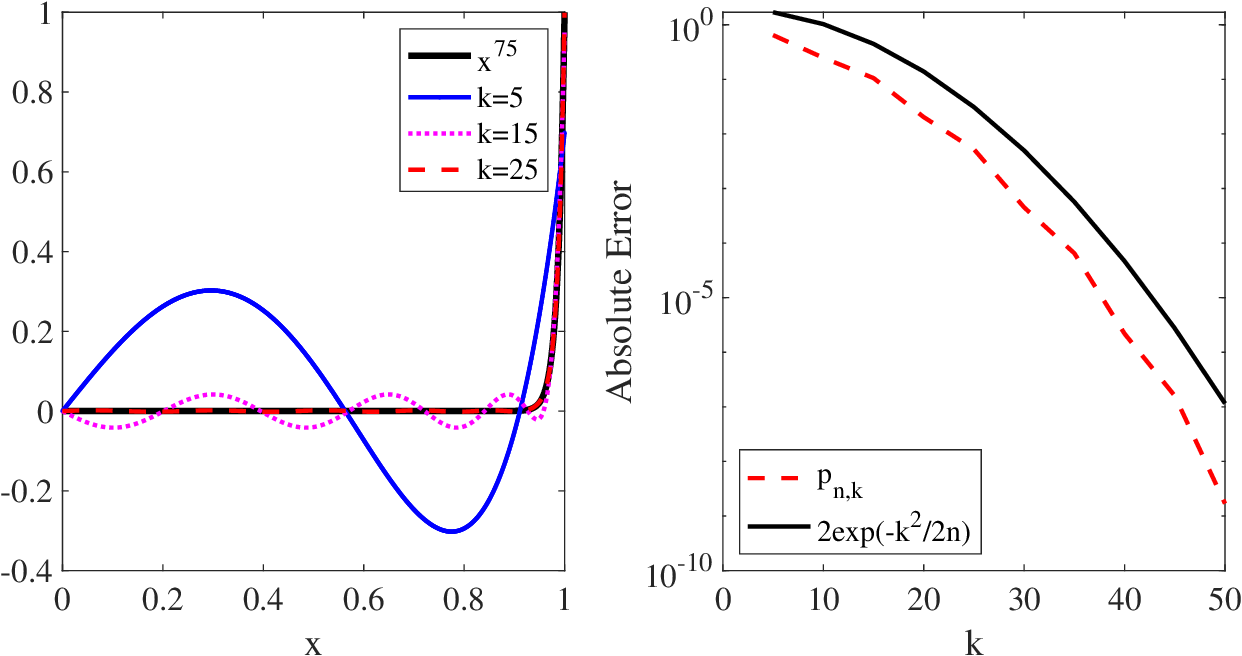}
    \caption{(left) Approximation of the monomial $x^{n}$ for $n=75$ by $\pi_k^*$, (right) $p_{n,k}$ and its upper bound $2\exp(-k^2/2n)$. The visualization is restricted to the interval $[0,1]$.}
    \label{fig:approx}
\end{figure}

 Polynomial and rational approximations to the monomial has received considerable attention in the approximation theory community, and surveys of various results can be found in~\cite{reddy1987approximations,nakatsukasa2018rational}. Polynomial approximations to high order monomials have many applications in numerical analysis.   This key insight was exploited by Cornelius Lanczos~\cite{lanczos1988applied} in his ``$\tau$-method'' for the numerical solution of differential equations. For a simulating discussion on this topic, please see~\cite{nakatsukasa2018rational}. In numerical linear algebra, this has been exploited to efficiently compute matrix powers and the Schatten p-norm of a matrix~\cite{avron2011randomized,dudley2020monte}.

 In this short note, we show to construct a polynomial approximation $x^n \approx \phi_k(x)$ using a truncated Chebyshev polynomial expansion. The error in the truncated representation equals the sum of the discarded coefficients and is precisely $p_{n,k}$. The polynomial $\phi_k$ and the resulting error can both be computed analytically and, therefore, is of great practical use.  We briefly review Chebyshev polynomials in Section~\ref{sec:cheby} and state and prove the main result in Section~\ref{sec:main}. In Section~\ref{sec:prob}, we explore probabilistic interpretations of $p_{n,k}$ and obtain bounds for partial sums of binomial coefficients.

\section{Chebyshev polynomials}\label{sec:cheby}
The Chebyshev polynomials of the first kind $T_n(x)$ for $n=0,1,2,\dots$ can be represented as 
\[T_n(x) = \cos(n\arccos{x}) \qquad x \in [-1,1]. \]
Starting with $T_0(x) = 1$, $T_1(x) = x$, the Chebyshev polynomials satisfy a recurrence relationship of the form  $T_{n+1}(x) = 2xT_n(x) - T_{n-1}(x)$ for $n\geq 1$. The Chebyshev polynomials are orthogonal with respect to the weighted inner product \[\langle u,v \rangle = \int_{-1}^1 w(x) u(x) v(x) dx \] 
where the weight function takes the form $w(x) = (1-x^2)^{-1/2}$. Any function $f \in C[-1,1]$ that is Lipschitz continuous can be represented in terms of a Chebyshev polynomial expansion of the form
\[ f(x) = \frac12 c_0 + \sum_{j=1}^\infty c_jT_j(x) \qquad x \in [-1,1], \]
where the coefficients $c_j$ are obtained as  $c_j = \frac{2}{{\pi}}\langle f(x), T_j(x)\rangle$ and the series is uniformly convergent. The monomial $x^n$ is rather special since it has the following exact representation in terms of the Chebyshev polynomials~\cite[Section 4]{cody1970survey} \begin{equation}\label{eqn:xncheby} x^n = \sum_{j=0}^{n}{}^{'} c_jT_j(x),\end{equation}
where ${}^{'}$ means the summand corresponding to $j=0$ is halved (if it appears) and  the coefficients $c_j$ for $j=0,\dots,n$ are
\begin{equation}\label{eqn:cj} c_j = \left\{ \begin{array}{ll} 2^{1-n}\binom{n}{(n-j)/2}  & n-j \text{ even} \\ 0 & \text{otherwise}.\end{array}  \right.\end{equation}
Equation~\eqref{eqn:xncheby} takes a more familiar form, when we consider the trigonometric perspective of Chebyshev polynomials. For example, the well-known trigonometric identity $\cos(3\theta) = 4\cos^3 \theta - 3\cos \theta$, can be arranged as \[ \cos^3 \theta = \frac{3}{4}\cos\theta + \frac{1}{4} \cos(3\theta) = \frac{1}{2^2} \left( \binom{3}{1} \cos\theta + \binom{3}{0} \cos(3\theta)\right). \]
With $x=\cos\theta$, we get $x^3 = 2^{-2} \binom{3}{1} T_1(x) + 2^{-2} \binom{3}{0} T_3(x)$. For completeness, we provide a derivation of~\eqref{eqn:xncheby} in Appendix~\ref{app:der}. It is important to note here that the series in~\eqref{eqn:xncheby} is finite, but can be truncated to obtain an accurate approximation; see Section~\ref{sec:main}.

Chebyshev polynomials have many applications in approximation theory and numerical analysis~\cite{trefethen2013approximation} but we limit ourselves to two such examples here. First, if the function is differentiable $r$ times or analytic, the Chebyshev coefficients exhibit decay (algebraic or geometric respectively). Therefore, the Chebyshev series can be truncated to obtain an  polynomial approximation of the function and the accuracy of the approximation depends on the rate of decay of the coefficients. Another application of Chebyshev polynomials is in the theory and practice of polynomial interpolation. The polynomial ${q}_{n-1}^*(x) = x^{n} - 2^{1-n}T_{n}(x)$ solves the minimax problem 
\begin{equation}\label{eqn:minmax} \min_{q \in \mathcal{P}_{n-1}} \|x^{n} - q(x)\|_\infty = 2^{1-n}. \end{equation}
Based on the minimax characterization, to interpolate a function over $[-1,1]$ by a polynomial of degree $n-1$, the function to be interpolated  should be evaluated at the roots of the Chebyshev polynomial $T_{n}$ given by the points $x_j = \cos\left(\frac{2j+1}{2n }\pi\right)$ for $j=0,\dots,n-1$.

\section{Main result}\label{sec:main}
We construct the polynomial approximation $ x^n \approx \phi_k(x)$ by truncating the Chebyshev polynomial expansion in~\eqref{eqn:xncheby} beyond the term $j=k$. That is \[\phi_k(x) := \sum_{j=0}^k{}^{'}c_jT_j(x).\]
Our main result is the following theorem, which quantifies the error in the polynomial approximation. The proof of this theorem is based on the expression in~\eqref{eqn:xncheby}. We believe this result is new.
\begin{theorem}\label{thm:main}
The error in the polynomial approximation $\phi_k(x)$ satisfies 
\[  \|x^n -  \phi_k(x)\|_\infty = p_{n,k}. \]
\end{theorem}

\begin{proof}
 From~\eqref{eqn:xncheby}, $x^{n} -  \phi_k(x) = \sum_{j=k+1}^nc_jT_j(x)$. Using triangle inequality, we find that $\|x^n-\phi_k(x)\|_\infty \leq \sum_{j=k+1}^n c_j$ since the coefficients are nonnegative and the Chebyshev polynomials are bounded as $|T_j(x)| \leq 1$. Substituting the coefficients $c_j$ from~\eqref{eqn:cj}, to get 
\begin{equation}
    \label{eqn:inter}
    \|x^n-\phi_k(x)\|_\infty \leq \frac{1}{2^{n-1}}\sum_{\substack{j=k+1 \\ n-j \text{ even}}}^n \binom{n }{(n-j)/2}.
\end{equation}
Using the properties of the binomial coefficients, the summation simplifies as
\[ 
\sum_{\substack{j=k+1 \\ n-j \text{ even}}}^n \binom{n }{(n-j)/2} =  \sum_{\substack{j=k+1 \\ n+j \text{ even}}}^n \binom{n }{(n+j)/2} =\sum_{j=\lfloor (n+k)/2\rfloor+1}^n \binom{n }{j}. \]
Plug this identity into~\eqref{eqn:inter} to get $\|x^n -  \phi_k(x)\|_\infty \leq p_{n,k}$. The bound is clearly achieved at $x= 1$, where all the Chebyshev polynomials take the value $1$. 
\end{proof}
This theorem shows that the polynomial approximation $\phi_k$ is nearly optimal, and the error due to this approximation is $p_{n,k}$. However, it is the optimal polynomial for the special case $k=n-1$. It is easy to see that $x^n - \phi_{n-1}(x) = 2^{1-n}T_n(x)$ and so $\phi_{n-1}$ is the same as the best polynomial approximation $q^*_{n-1}$ in~\eqref{eqn:minmax}. For $k < n-1$, from~\eqref{eqn:newmanrivlin} and Theorem~\ref{thm:main}
\[ \frac{ \|x^n-\phi_k(x)\|_\infty}{4e} \leq  \|x^n-\pi_k^*(x)\|_\infty \leq  \|x^n-\phi_k(x)\|_\infty, \]
so that the error in the Chebyshev polynomial approximation is suboptimal by at most the factor $4e \approx 10.87$. Therefore, by using $\phi_k$ we lose only one significant digit of accuracy compared to $\pi^*_k$. 

\section{A probabilistic digression}\label{sec:prob}
In Section~\ref{sec:intro}, we saw that the error in the monomial approximation depends on  $p_{n,k}$. Since $p_{n,k}$ depends on the sum of binomial coefficients, it has a probabilistic interpretation. Newman and Rivlin~\cite{newman1976approximation} observed that if a fair coin is tossed $n$ times, $p_{n,k}$ is the probability that the magnitude of the difference between the number of heads and the number of tails exceeds $k$. They used this insight along with the de Moivre-Laplace theorem~\cite[Section 1.3]{vershynin2018high} (which is a special case of the Central Limit Theorem) to obtain the approximation $p_{n,k} \approx 2 \text{erfc}(k/\sqrt{n})$.

To convert this into a rigorous inequality for $p_{n,k}$ we use a different tool from probability, namely, concentration inequalities. The inequalities are useful in quantifying how much a random variable deviates from its mean.  We start with the following alternative interpretation for $p_{n,k}$: it is twice the probability that greater than   $\lfloor (n+k)/2 \rfloor$ coin tosses result in heads (or equivalently tails). We associate each coin toss with an independent Bernoulli random variable $X_i$ with parameter $p=1/2$ since the coin is fair. The random variable $X = \sum_{i=1}^nX_i$ has the Binomial distribution with parameters $n$ and $p$. Then, 
\[ p_{n,k} =  2\mathbb{P}\left( \lfloor (n+k)/2 \rfloor +1 \leq X \leq n\right)  \leq 2\mathbb{P}\left(X  \geq (n+k)/2 \right). \] 
Since $X$ has the Binomial distribution, we can once again use the de Moivre-Laplace theorem, to say that as $n\rightarrow \infty$, 
\[ \frac{X - np}{\sqrt{np(1-p)}} \longrightarrow \mathcal{N}(0,1), \qquad \text{in distribution}.\]
  Roughly speaking, this theorem says that $X$ behaves as a normal random variable with mean $n/2$ and variance $n/4$. Since the tails of normal distributions decay exponentially, we expect that $X$ lies in the range $\frac{n}{2} \pm 1.96 \sqrt{\frac{n}{4}}$ with nearly $95\%$ probability; alternatively, the probability that it is outside this range is very small. To make this more precise, we apply Hoeffding's concentration inequality~\cite[Theorem 2.2.6]{vershynin2018high}, to obtain
\[  \mathbb{P}\left(X  \geq (n+k)/2 \right) = \mathbb{P}\left(X - \mathbb{E}[X] \geq k/2 \right) \leq \exp\left(-\frac{k^2}{2n}\right). \]
This gives our desired bound $p_{n,k} \leq 2\exp(-{k^2}/{2n})$.

We can use a similar technique to prove the following result which may be of independent interest. If $0 \leq k \leq n/2$, then 
\[ \sum_{j=0}^k \binom{n}{j} \leq 2^n \exp\left( - \frac{(n-2k)^2}{2n}\right).\]
 Other concentration inequalities such as Chernoff and Bernstein (see~\cite[Chapter 2]{vershynin2018high}) also give equally interesting bounds. We invite the reader to explore such results.

\section{Acknowledgements}
The author would like to thank Alen Alexanderian, Ethan Dudley, Ivy Huang, Ilse Ipsen, and Nathan Reading for comments and feedback. The work was supported by the National Science Foundation through the grants DMS-1745654 and DMS-1845406.
\section{Declaration of Interest}
The author has no relevant financial or non-financial competing interests to report.

\appendix
\section{Derivation of the monomial expansion}\label{app:der}
In this appendix, we provide a short derivation of~\eqref{eqn:xncheby}. We take $x = \cos \theta$ and write 
\[  \cos^n\theta = \left(\frac{e^{i\theta} + e^{-i\theta}}{2}\right) = \frac{1}{2^n} \sum_{j=0}^n\binom{n}{j}e^{ij\theta}e^{-i(n-j)\theta} = \frac{1}{2^n} \sum_{j=0}^n\binom{n}{j}e^{i(2j-n)\theta}. \]
We have used Euler's formula and the binomial theorem. At this point, the derivation splits into two different paths:
\begin{description}
\item[1. $n$ is odd] 
\[ \begin{aligned}\cos^n\theta  =  & \> \frac{1}{2^n} \left( \sum_{j=0}^{(n-1)/2} \binom{n}{j}e^{i(2j-n)\theta}  + \sum_{j=(n+1)/2}^{n} \binom{n}{j}e^{i(2j-n)\theta}\right) \\
= & \> \frac{1}{2^{n-1}} \sum_{j=0}^{(n-1)/2}\binom{n}{j}\cos((n-2j)\theta) \\
= & \> \frac{1}{2^{n-1}}\sum_{\substack{j=0 \\ n-j \text{ even}}}^n \binom{n}{(n-j)/2} \cos(j\theta) .\end{aligned}\]
\item [2. $n$ is even] 
\[ \begin{aligned}\cos^n\theta  =  & \> \frac{1}{2^n} \left( \sum_{j=0}^{n/2-1} \binom{n}{j}e^{i(2j-n)\theta}  + \binom{n}{n/2} +  \sum_{j=n/2+1}^{n} \binom{n}{j}e^{i(2j-n)\theta}\right) \\
= & \>  \frac{1}{2^n} \binom{n}{n/2} + \frac{1}{2^{n-1}} \sum_{j=0}^{n/2-1}\binom{n}{j}\cos((n-2j)\theta) \\
= & \> \frac{1}{2^{n-1}}\sum_{\substack{j=0 \\ n-j \text{ even}}}^n{}^{'} \binom{n}{(n-j)/2} \cos(j\theta)  .\end{aligned}\] 
\end{description} 
In either case, substitute $x = \cos\theta$ and $T_j(x) = \cos(j\theta)$ to complete the derivation.

\bibliography{refs}
\bibliographystyle{abbrv}
\end{document}